\documentclass[11pt]{article}%
\usepackage{amsmath}
\usepackage{amsfonts}
\usepackage{amssymb}
\usepackage{graphicx}
\usepackage{subfig}%
\providecommand{\U}[1]{\protect\rule{.1in}{.1in}}
\newtheorem{theorem}{Theorem}[section]

\newtheorem{condition}{Assumption}[section]

\newtheorem{corollary}{Corollary}[section]

\newtheorem{proposition}{Proposition}[section]
\newtheorem{remark}{Remark}[section]

\newenvironment{proof}[1][Proof]{\noindent\textbf{#1.} }{\ \rule{0.5em}{0.5em}}
\numberwithin{equation}{section}

\usepackage{color}

\def\fr{\frac} \def\ol{\overline} \def\sn{spectrally negative }
\def\la{\label} \def\dd{drawdown } \def\Eq{\Leftrightarrow} 

\newcommand{\ninseps}[3]{
\begin{figure}[h]
\begin{center}
 \scalebox{#3}{\includegraphics{#1}}
\end{center}

\vspace{-0.2cm}
\caption{\hspace{0.25cm}#2\label{f:#1}}
\end{figure}
}   
   
 \def\be{\begin{equation}} \def\ee{\end{equation}}
    
\def\dd{drawdown } \def\sn{spectrally negative }

    \def\bc{\begin{cases}
  }     
\def\ec{\end{cases}}
\long\def\symbolfootnote[#1]#2{
\begingroup
\def\thefootnote{\fnsymbol{footnote}}\footnote[#1]{#2}
\endgroup}  \def\Eq{\Leftrightarrow}

\newtheorem{Thm}{Theorem}

\newtheorem{Lem}{Lemma}
\newtheorem{Rem}[Lem]{Remark}
\newtheorem{Exe}[Lem]{Exercice}

\def\beXe{\begin{Exe}} \def\eeXe{\end{Exe}}
\def\eeD{\end{Def}} \def\beD{\begin{Def}}
\def\beXa{\begin{Exa}} \def\eeXa{\end{Exa}}
\def\beR{\begin{Rem}} \def\eeR{\end{Rem}}
\def\beL{\begin{Lem}} \def\eeL{\end{Lem}}
\newtheorem{Pro}[Lem]{Proposition}
\def\beP{\begin{Pro}} \def\eeP{\end{Pro}}
\def\beC{\begin{Cor}} 
\def\beT{\begin{Thm}}
  \def\eeT{\end{Thm}}
\usepackage{cite}

\begin{document}

\title{General drawdown of general tax model in a time-homogeneous Markov framework}
\author{Florin Avram\thanks{Laboratoire de Math\'ematiques Appliqu\'ees, Universit\'e de Pau, France}
	\and Bin Li\thanks{Department of Statistics and Actuarial Science, University of
		Waterloo, Waterloo, ON, N2L 3G1, Canada}
	\and Shu Li\thanks{Department of Statistical and Actuarial Sciences, Western University, London, ON, N6A 5B7, Canada}}
\maketitle

\begin{abstract}
Drawdown/regret times feature prominently in optimal stopping problems, in  statistics (CUSUM procedure) and in mathematical finance (Russian options).  Recently it was discovered that a first passage theory with  general drawdown times, which generalize classic ruin times, may be explicitly developed for spectrally negative L\'evy processes -- see Avram, Vu,  Zhou(2017), Li, Vu,  Zhou(2017). In this paper, we further examine  general drawdown related quantities for taxed time-homogeneous Markov processes,  using the pathwise connection between general drawdown and tax process.

\end{abstract}

\section{Introduction}

Our paper is part of a larger program to improve the control of a reserves/risk
 process $X$. The rough idea is that when below low levels $a$, the reserves should be replenished at some cost, and when above high levels $b$, the reserves should be invested to yield dividends -- see for example \cite{AA10}. The low levels first considered historically have been those of $X$, but  one may equally consider low levels of the drawdown/regret/process reflected at the maximum,
defined by
\[
  D_t=\overline { X}_t-X_t , \quad \overline X_t:= \sup_{0 \leq s\leq t} X_s ,
\]
which turn out to be of interest in several problems in statistics, mathematical finance
and risk theory \cite{P54,T75,L77,SS93,AKP04,C14,MP12,LLL15,LLZ17b,AVZ17,LVZ17,ASVA18}. The book\cite{Z18} summarizes most of the recent developments on drawdown.

Assume from now on that our underlying  process $X$ is time-homogeneous and Markovian.
 The first passage times of $X$ across
a level $x\in%
\mathbb{R}
$ are denoted by
\[
\tau_{x}^{+}=\inf\left\{  t\geq0:X_{t}>x\right\}  \text{ and }\tau_{x}%
^{-}=\inf\left\{  t\geq0:X_{t}<x\right\}  .
\]

For simplicity, we
assume $X$ is {upward skip-free}. Moreover, we assume $X$ is regular in the sense that $\mathbb{P}_{y}(\tau
_{x}^{+(-)}<\infty)>0$, for all $x,y\in\mathbb{R}$.

Instrumental in achieving the control of one dimensional risk processes   are the distributions  of the two-sided smooth and non-smooth first passage times  from a bounded  interval $[u,v]$. For upward skip-free processes, it turns out easier to study  the corresponding  Laplace transforms:
\begin{align}
B^{(q)}(x;u,v) &  :=\mathbb{E}_{x}\left[  e^{-q\tau_{v}^{+}}1_{\left\{
\tau_{v}^{+}<\tau_{u}^{-}\right\}  }\right]  ,\la{B}\\
C^{(q,s)}(x;u,v) &  :=\mathbb{E}_{x}\left[  e^{-q\tau_{u}^{-}-s(u-X_{\tau
_{u}^{-}})}1_{\left\{  \tau_{u}^{-}<\tau_{v}^{+}\right\}  }\right]  ,\la{C}
\end{align}
where   $q,s\geq0$, and $u\leq x\leq v$.
Indeed, for L\'evy processes for example it holds that:
\[ B^{(q)}(x;u,v)    =\fr{W_q(x-u)}{W_q(v-u)},\]
where $W_q(x)$ is called the scale function \cite{S76, B98,K14}, and  for some non-homogeneous \sn Markov processes \cite{CPRY17} a similar formula  holds
\[ B^{(q)}(x;u,v)    =\fr{W_q(x;u)}{W_q(v;u)} ,\]
where now the newly defined scale function naturally depends on the two variables $x,u$.

Several control problems for $(X,D)$ are known to reduce  to the  study of the process $\ol X_t$ with all its negative excursions excised, which turns out to
be a deterministic process, killed at a random time \cite{ABBR09,AACI14}--see Figure \ref{f:samplepath3} below. This supports
the parallel fundamental idea  of \cite{LLZ17b} to base the study of $(X,D)$  on the existence
of two differential parameters.
\begin{condition}
\label{asp}For all $q,s\geq0$ {and $u\leq x$ fixed,  assume that  $B^{(q)}(x;u,v)$ and
$C^{(q,s)}(x;u,v)$ are differentiable in $v$ at  $v=x$} 
and denote
\[
\left.  \frac{\partial B^{(q)}(x;u,v)}{\partial v}\right\vert _{v=x}%
=-b_{u}^{(q)}(x)\text{ and }\left.  \frac{\partial C^{(q,s)}(x;u,v)}%
{\partial v}\right\vert _{v=x}=c_{u}^{(q,s)}(x).
\]

\end{condition}

A necessary condition for Assumption \ref{asp} to hold is that
\[
\tau_x^+=0 \mbox{ and } X_{\tau_x^+}=x, \mathbb{P}_{x}
-a.s. \mbox{ for all }x\in \mathbb{R}.
\]

To understand the joint dynamics of two dimensional process $t\mapsto (X_t,D_t)$, it is useful
to look at Figure \ref{f:samplepath3}, reproduced from \cite{ASVA18}, which depicts a sample path of $(X,D)$, where $X$ is chosen to be the
standard Brownian motion and the exit region is $R=[-6,7]\times[0,10]$.
\ninseps{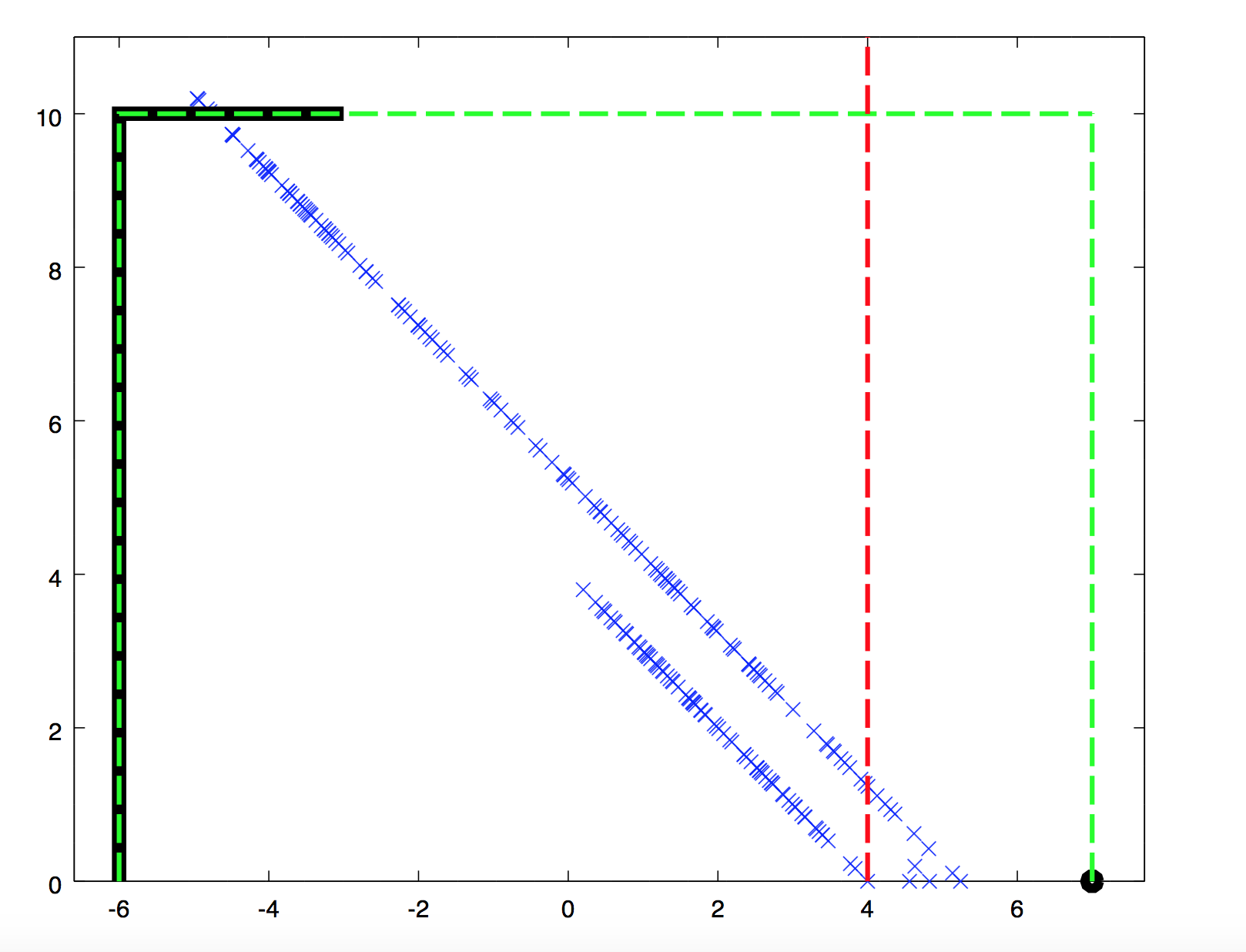}{A sample path of $(X,D)$ (sampled at time step $\Delta t = 0.1$) when $X$ is a standard Brownian motion with $X_0 =a+d=4$, and the region $R$ with $d=10$, $a=-6$ and $b=7$; the dark shaded region shows the possible points of exit of $(X,D)$ from $R=[-6,7]\times[0,10]$}{0.45}
As is clear from the figure and
from its definition, the process $(X,D)$ has very particular
dynamics on $R$: away from the boundary
$\partial_1 := \{x \in {\mathbb R} \times {\mathbb R}_+: x_2 = 0 \}$
it oscillates on the line segment $L_{\ol{X}_t}$ where, for $c \in {\mathbb R}$,
$L_c := \{x \in {\mathbb R} \times {\mathbb R}_+: x_1 + x_2 = c \}$. 
{These oblique lines
represent each a negative excursion}.
{On $\partial_1$, we observe the
evolution of the process $\ol X_t$ with all its negative excursions excised};
 as $\ol {X}_t$ increases, the line
segment $L_{\ol{X}_t}$ on which $(X,D)$ oscillates 
{during a negative excursion} advances continuously to the right.

To fully specify the process $\ol X_t$ with  its negative excursions excised, we must give a rule for killing a negative excursion; two classic choices  are $X_t <a$ (ruin stopping) and $D_t >d$ (\dd stopping), which are the left and upper boundaries in Figure \ref{f:samplepath3}, respectively. A linear combination of these, translating into an oblique upper boundary, has been studied in \cite{AVZ17}.

In our paper we consider  more general upper boundaries,  which
include the previous works as particular cases. Following \cite{LVZ17}, we consider stopping times
\[
\tau_{f}=\inf\left\{  t\geq0:X_{t}<f(\overline X_{t})\right\}  =\inf\{t\geq0:Y_{t}>0\} ,
\]
where $$Y_{t}=f(\overline X_{t})-X_{t}=D_t- \ol f(\ol X_t), \; {t\geq0}$$ will be called
a general drawdown
process. Here  $\ol f(m):=m-f(m)$, and $f$ must be 
{nondecreasing} such that
\[
f(x)<x \Eq \ol f(x) >0,\quad x\in%
\mathbb{R}
\text{.}%
\]
Note that we have $Y_{0}=f(X_{0})-X_{0}<0$.

General \dd times include many important particular subcases which have been extensively studied in the literature:
\begin{enumerate}

\item If $f(x)=0$, $\tau_f=\tau_{0,0}$ is the ruin time.

\item If $f(x)=x-d$, $\tau_f=\tau_{1,d}$ is the classic drawdown time.

\item If $f(x)=\xi x, \xi <1$, when $\tau_f=\tau_{\xi,0}$ is the proportional drawdown time.
\item  If \begin{equation} f(x)=\xi x - d \Eq \bar f(m)=m-f(m)=(1-\xi) m + d \label{xi}, \; \xi\in {(-\infty, 1]}, d \geq 0, (1-\xi)>0, \end{equation}
the   corresponding drawdown time is
\begin{equation}
 \tau_f=\tau_{\xi,d}=\inf \left\{t\geq 0: X_t \leq
 \xi\overline {X}_t-d\right\} =\inf \left\{t\geq 0: \overline {X}_t -X_t >
 (1-\xi)\overline {X}_t+d\right\}.\label{ddl}
 \end{equation}

This is called the affine drawdown studied in \cite{AVZ17}. It turns out that this extension complicates only slightly the classic drawdown results, while allowing  treating simultaneously times cases 2 and 3.

\item Nonlinear drawdown times emerged in \cite{L77} and  were used by Az\'ema and Yor \cite{AY79} to  provide a solution of the Skorokhod problem of stopping a {Brownian motion} to obtain a  given desired centered marginal measure.

\end{enumerate}

{\bf Contents}. Below, we extend first the  general \dd results  of \cite{LVZ17} from \sn
 L\'evy processes to \sn time-homogeneous Markov processes -- see Section \ref{s:dd}. Then, in Section \ref{s:tax} we allow also for the possibility
 of general taxation. The method of proof involves a nontrivial use of the ``differential exit problems" of \cite{LLZ17b}.
The results in Section \ref{s:dd} are applied  in the three particular cases in which the ``differential exit parameters" of \cite{LLZ17b} are analytically computable: \sn
 L\'evy processes and diffusions. A third example, which is illustrated in \cite{LLZ17b}, is of Ornstein-Uhlenbeck-type processes with exponential jumps.

\section{Main results of general drawdown in the time-homogeneous Markov process}\label{s:dd}

The following pathwise inequalities are central to the construction of tight
bounds for the joint law of the triplet $(\tau_{f},\overline X_{\tau_{f}},Y_{\tau_{f}})$.

\begin{proposition}
\label{prop path}For $q,s\geq0$, $x\in\mathbb{R}$ and $\varepsilon>0$, we have
$\mathbb{P}_{x}$-a.s.
\begin{equation}
1_{\{\tau_{x+\varepsilon}^{+}<\tau_{f(x+\varepsilon)}^{-}\}}\leq1_{\{\tau_{x+\varepsilon}^{+}<\tau_{f}\}}\leq1_{\{\tau_{x+\varepsilon}^{+}<\tau_{f(x)}^{-}\}}, \label{eq.up}%
\end{equation}
and
\begin{align}
e^{-q\tau_{f}-sY_{\tau_{f}}}1_{\left\{ \tau_{f}<\tau_{x+\varepsilon}^{+}\right\}  }  &  \geq e^{-q\tau_{f(x)}^{-}-s(f(x+\varepsilon)-X_{\tau_{f(x)}^{-}})}1_{\{\tau_{f(x)}^{-}<\tau_{x+\varepsilon
}^{+}\}},\label{eq.down1}\\
e^{-q\tau_{f}-sY_{\tau_{f}}}1_{\left\{\tau_{f}<\tau_{x+\varepsilon}^{+}\right\}  }  &  \leq e^{-q\tau_{f(x+\varepsilon)}^{-}-s(f(x)-X_{\tau_{f(x+\varepsilon)}^{-}})}1_{\{\tau_{f(x+\varepsilon)}^{-}<\tau_{x+\varepsilon}^{+}\}}. \label{eq.down2}%
\end{align}

\end{proposition}

\begin{proof}
By analyzing the sample paths of $X$, it is easy to see that $\tau_{f}\leq
\tau_{f(x)}^{-}\quad\mathbb{P}_{x}$-a.s. Thus, $\mathbb{P}_{x}$-a.s. we have
\begin{equation}
(\tau_{x+\varepsilon}^{+}<\tau_{f})=(\tau_{x+\varepsilon}^{+}<\tau_{f}\leq
\tau_{f(x)}^{-})\subset(\tau_{x+\varepsilon}^{+}<\tau_{f(x)}^{-})\nonumber
\end{equation}
and%
\begin{equation}
(\tau_{x+\varepsilon}^{+}<\tau_{f(x+\varepsilon)}^{-})=(\tau_{x+\varepsilon}^{+}<\tau_{f(x+\varepsilon
)}^{-},\tau_{x+\varepsilon}^{+}<\tau_{f})\subset( \tau_{x+\varepsilon}^{+}<\tau_{f}),\nonumber
\end{equation}
which immediately implies (\ref{eq.up}).

On the other hand, by using the same argument, we have, $\mathbb{P}_{x}%
$-a.s.,
\begin{equation}
(\tau_{f(x)}^{-}<\tau_{x+\varepsilon}^{+})=(\tau_{f}\leq \tau_{f(x)}^{-}<\tau_{x+\varepsilon}^{+})\subset(\tau_{f}<\tau_{x+\varepsilon}^{+}), \label{w1}%
\end{equation}
and%
\begin{equation}
(\tau_{f}<\tau_{x+\varepsilon}^{+})=(\tau_{f(x+\varepsilon)}^{-}\leq\tau_{f}<\tau_{x+\varepsilon}^{+})\subset
(\tau_{f(x+\varepsilon)}^{-}<\tau_{x+\varepsilon
}^{+}). \label{w2}%
\end{equation}
For any path $\omega\in(\tau_{f(x)}^{-}<\tau_{x+\varepsilon}%
^{+})$, we know from (\ref{w1}) that $\omega\in(\tau
_{f}\leq \tau_{f(x)}^{-}<\tau_{x+\varepsilon}^{+})$. This implies $\overline X_{\tau_{f}%
}(\omega)\leq x+\varepsilon$ and $X_{\tau_{f}}(\omega)\geq X_{\tau_{f(x)}^{-}%
}(\omega)$, which further entails that $Y_{\tau_{f}}(\omega)=f(\overline X_{\tau_{f}%
}(\omega))-X_{\tau_{f}}(\omega)\leq f(x+\varepsilon)-X_{\tau_{f(x)}^{-}}(\omega
)$. Therefore, by the above analysis and (\ref{w1}),$\quad\mathbb{P}_{x}%
$-a.s.,
\[
e^{-q\tau_{f(x)}^{-}-s(f(x+\varepsilon)-X_{\tau_{f(x)}^{-}})}1_{\left\{
\tau_{f(x)}^{-}<\tau_{x+\varepsilon}^{+}\right\}  }\leq
e^{-q\tau_{f}-sY_{\tau_{f}}}1_{\left\{ \tau_{f}%
<\tau_{x+\varepsilon}^{+}\right\}  }%
\]
which naturally leads to (\ref{eq.down1}).

Similarly, for any sample path $\omega\in(\tau_{f}%
<\tau_{x+\varepsilon}^{+})$, we know from (\ref{w2}) that $\omega\in(\tau_{f(x+\varepsilon)}^{-}\leq\tau_{f}<\tau_{x+\varepsilon}^{+})$,
which implies that $f(x)-X_{\tau_{f(x+\varepsilon)}^{-}}(\omega)\leq
Y_{\tau_{f(x+\varepsilon)}^{-}}(\omega)\leq Y_{\tau_{f}}(\omega)$. Here the last
inequality is because $Y_{\tau_{f(x+\varepsilon)}^{-}}(\omega)\leq0\leq
Y_{\tau_{f}}(\omega)$ if $\tau_{f(x+\varepsilon)}^{-}<\tau_{f}$, and
$Y_{\tau_{f(x+\varepsilon)}^{-}}(\omega)=Y_{\tau_{f}}(\omega)$ if
$\tau_{f(x+\varepsilon)}^{-}=\tau_{f}$. Therefore, we obtain,$\quad\mathbb{P}%
_{x}$-a.s.,
\[
e^{-q\tau_{f}-sY_{\tau_{f}}}1_{\left\{\tau_{f}%
<\tau_{x+\varepsilon}^{+}\right\}  }\leq e^{-q\tau_{f(x+\varepsilon)}^{-}%
-s(f(x)-X_{\tau_{f(x+\varepsilon)}^{-}})}1_{\{\tau_{f(x+\varepsilon)}^{-}<\tau_{x+\varepsilon}^{+}\}},
\]
which proves (\ref{eq.down2}).\bigskip
\end{proof}

By Proposition \ref{prop path}, we easily obtain the following useful estimates.

\begin{corollary}
\label{cor bd}For $q,s\geq0$, $x\in%
\mathbb{R}
$ and $\varepsilon>0$,%
\[
B^{(q)}(x;f(x+\varepsilon),x+\varepsilon)\leq\mathbb{E}_{x}\left[  e^{-q\tau_{x+\varepsilon}^{+}}1_{\{\tau
_{x+\varepsilon}^{+}<\tau_{f}\}}\right]  \leq B^{(q)}(x;f(x),x+\varepsilon),
\]
and%
\begin{align*}
\mathbb{E}_{x}\left[  e^{-q\tau_{f}-sY_{\tau_{f}}}1_{\left\{  \tau
_{f}<\tau_{x+\varepsilon}^{+}\right\}  }\right]    & \leq
e^{s(f(x+\varepsilon)-f(x))}C^{(q,s)}%
(x;f(x+\varepsilon),x+\varepsilon)\\
\mathbb{E}_{x}\left[  e^{-q\tau_{f}-sY_{\tau_{f}}}1_{\left\{  \tau_{f}<\tau_{x+\varepsilon}^{+}\right\}  }\right]    & \geq
e^{-s(f(x+\varepsilon)-f(x))}C^{(q,s)}(x;f(x),x+\varepsilon)
\end{align*}

\end{corollary}

Next we present our main results of the general drawdown.

\begin{theorem}
\label{thm markov}Consider an upward skip-free time-homogeneous Markov process
$X$ such that Assumption \ref{asp} holds. For $q,s\geq0$ and $x<K\in
\mathbb{R}$, we have
\begin{align}
\mathbb{E}_{x}\left[  e^{-q\tau_{K}^{+}}1_{\{\tau_{K}^{+}<\tau_{f}%
\}}\right]    & =e^{-\int_{x}^{K}b_{f}^{(q)}(z)\mathrm{d}%
z}, \label{exitK}\\
\mathbb{E}_{x}\left[  e^{-q\tau_{f}-sY_{\tau_{f}}}1_{\{\overline
{X}_{\tau_{f}}\leq K\}}\right]    & =\int_{x}^{K}e^{-\int_{x}^{y}b_{f}^{(q)}(z)\mathrm{d}z}c_{f}^{(q,s)}(y)\mathrm{d}y. \label{ddfK}
\end{align}

\end{theorem}

\begin{proof}
Let%
\[
g(x)=\mathbb{E}_{x}\left[  e^{-q\tau_{K}^{+}}1_{\{\tau_{K}^{+}<\tau_{f}\}}\right]  ,\quad x<K.
\]
By the strong Markov property of $X$ at maxima, for any $X_{0}=x\leq y<K$ and
$0<\varepsilon<K-y$, we have
\[
g(y)=\mathbb{E}_{y}\left[  e^{-q\tau_{y+\varepsilon}^{+}}1_{\{\tau
_{y+\varepsilon}^{+}<\tau_{f}\}}\right]  g(y+\varepsilon).
\]
By Corollary \ref{cor bd}, it follows that
\[
B^{(q)}(y;f(y+\varepsilon),y+\varepsilon)g(y+\varepsilon)\leq g(y)\leq B^{(q)}(y;f(y),y+\varepsilon)g(y+\varepsilon).
\]
It follows that
\[
\left\{
\begin{array}
[c]{l}%
g(y+\varepsilon)-g(y)\leq\left[  1-B^{(q)}(y;f(y+\varepsilon),y+\varepsilon)\right]  g(y+\varepsilon)\\
g(y+\varepsilon)-g(y)\geq\left[  1-B^{(q)}(y;f(y),y+\varepsilon)\right]  g(y+\varepsilon)
\end{array}
\right.
\]
By Assumption \ref{asp}, it follows that
\[
g^{\prime}(y)=b_{f}^{(q)}(y)g(y),\quad y<K,
\]
with boundary condition $g(K)=1$. Thus,
\[
g(x)=e^{-\int_{x}^{K}b_{f}^{(q)}(z)\mathrm{d}z},\quad x<K.
\]

Similarly, let
\[
h(x)=\mathbb{E}_{x}\left[  e^{-q\tau_{f}-sY_{\tau_{f}}}1_{\{\overline
{X}_{\tau_{f}}\leq K\}}\right]  ,\quad x<K.
\]
By the strong Markov property of $X$ at maxima, for any $X_{0}=x\leq y<K$ and
$0<\varepsilon<K-y$, we have
\[
h(y)=\mathbb{E}_{y}\left[  e^{-q\tau_{f}-sY_{\tau_{f}}}1_{\left\{
\tau_{f}<\tau_{y+\varepsilon}^{+}\right\}  }\right]  +\mathbb{E}%
_{y}\left[  e^{-q\tau_{y+\varepsilon}^{+}}1_{\{\tau_{y+\varepsilon}%
^{+}<\tau_{f}\}}\right]  h(y+\varepsilon).
\]
By Corollary \ref{cor bd}, it follows that%
\[
\left\{
\begin{array}
[c]{c}%
h(y)\leq e^{s(f(y+\varepsilon)-f(y))}C^{(q,s)}(y;f(y+\varepsilon),y+\varepsilon)+B^{(q)}(y;f(y),y+\varepsilon)h(y+\varepsilon),\\
h(y)\geq e^{-s(f(y+\varepsilon)-f(y))}C^{(q,s)}(y;f(y),y+\varepsilon)+B^{(q)}(y;f(y+\varepsilon),y+\varepsilon)h(y+\varepsilon).
\end{array}
\right.
\]
It follows that%
\[
\left\{
\begin{array}
[c]{l}%
h(y+\varepsilon)-h(y)\geq-e^{s(f(y+\varepsilon)-f(y))}C^{(q,s)}(y;f(y+\varepsilon),y+\varepsilon)+\left[  1-B^{(q)}(y;f(y),y+\varepsilon)\right]  h(y+\varepsilon),\\
h(y+\varepsilon)-h(y)\leq-e^{-s(f(y+\varepsilon)-f(y))}C^{(q,s)}(y;f(y),y+\varepsilon)+\left[  1-B^{(q)}%
(y;f(y+\varepsilon),y+\varepsilon)\right]
h(y+\varepsilon).
\end{array}
\right.
\]
By Assumption \ref{asp}, we deduce that
\[
h^{\prime}(y)=-c_{f}^{(q,s)}(y)+b_{f}^{(q)}(y)h(y),\quad y<K,
\]
with boundary condition $h(K)=0$. Therefore,%
\[
h(x)=\int_{x}^{K}e^{-\int_{x}^{y}b_{f}^{(q)}(z)\mathrm{d}z}c_{f}^{(q,s)}(y)\mathrm{d}y,\quad x<K.
\]
This ends the proof.
\end{proof}

\section{Extension to the general loss-carry-forward taxation model}\label{s:tax}

The loss-carry-forward taxation model is first proposed by Albrecher and Hipp \cite{AH07} under the compound Poisson model. It has been extended to the spectrally negative L\'evy model by Albrecher et al. \cite{ARZ08}, the time-homogeneous diffusion model by Li et al. \cite{LTZ13}, and the Markov additive model by Albrecher \cite{AACI14}.

In this section, we will further incorporate the general taxation proposed by Kyprianou and Zhou \cite{KZ09}. As our underlying model is upward skip-free Markov processes, our results will generalize \cite{ARZ08}, \cite{KZ09}, and \cite{LTZ13}. It is worth to mention that the methodologies adopt in these previous works are quite different, while this paper utilizes a unified and also more direct approach.

Consider a loss-carry-forward type tax strategy, where the tax payment is made whenever the surplus process reaches a new running maximum, (e.g., Kyprianou and Zhou \cite{KZ09})
\begin{equation}
\mathrm{d}U_{t}=\mathrm{d}X_{t}-\gamma(\overline{X}_{t})\mathrm{d}\overline
{X}_{t}, \label{taxmodel}
\end{equation}
where $\gamma:[0,\infty)\rightarrow[0,\infty)$ is a measurable function. Note that for $\gamma:[0,\infty)\rightarrow(1,\infty)$, it is heavy-perturbation regime; while for $\gamma:[0,\infty)\rightarrow[0,1)$, it is light-perturbation regime; $\gamma(x)=1_{\{ x\ge a\}}$ corresponds to a reflection strategy, which sits between the previous two regimes, see, e.g., Kyprianou \cite{K14}. In what follows, we only consider the light-perturbation case with a non-decreasing function $\gamma(\cdot)$, and in addition, we assume the following condition holds:
\[
\int_x^{\infty}(1-\gamma(s))ds=\infty.
\]

For $X_{0}=x$, define
\[
\overline{\gamma}_x(y):=y-\int_{x}^{y}\gamma(z)\mathrm{d}z=x+\int_{x}%
^{y}(1-\gamma(z))\mathrm{d}z,\quad y\geq x,
\]
which is strictly increasing and continuous with $\overline{\gamma}(x)=x$, and let $\gamma_x(y):=y-\overline \gamma_x(y)$.

The first passage times of $U$ are defined in the same manner, i.e.,
\[
\tau_{x}^{U,+}=\inf\left\{  t\geq0:U_{t}>x\right\}  \text{ and }\tau_{x}%
^{U,-}=\inf\left\{  t\geq0:U_{t}<x\right\}  .
\]
Note that, conditional on $X_{0}=U_{0}=x$, for any $y\geq x$, we have
\[
\overline{U}_{t}=\overline{\gamma}_x(\overline{X}_{t})\text{ and }\tau_{y}%
^{U,+}=\tau_{\overline{\gamma}^{-1}(y)}^{+}\text{.}%
\]

The general drawdown process of the tax model $U$ is denoted by $Y^U=(Y^U_{t})_{t\geq0}$ with
\[
Y^U_{t}=f(\overline{U}_{t})-U_{t},
\]
where $\overline{U}_{t}=\sup_{0\leq s\leq t}U_{t}$ and $f$ is an increasing
function\ such that%
\[
f(x)<x,\quad\text{for all }x\in%
\mathbb{R}
\text{.}%
\]
Hence, $Y^U_{0}=f(\overline{U}_{0})-U_{0}<0$. The time of general drawdown is
defined by
\[
\sigma_{f}=\inf\left\{  t\geq0:Y^U_{t}>0\right\}  =\inf\left\{  t\geq
0:U_{t}<f(\overline{U}_{t})\right\}  .
\]

Actually, from the general drawdown results for a general model $X$ in Theorem \ref{thm markov}, by noting the pathwise connection between $X$ and $U$, one can easily find the general drawdown results for a general tax model $U$ associated with the time-homogeneous Markov process $X$.

In the following, we first provide some time correspondences between processes $U$ and $X$. Given $X_0=U_0=x$, in the light-perturbation case,
\begin{itemize}

\item[(i)]
\begin{equation}
\tau_b^{U,+}=\tau_{\overline\gamma_x^{-1}(b)}^+,\quad a.s., \label{time-i}
\end{equation}
since $\overline U_t=\overline X_t- \int_{(0,t]} \gamma(\overline X_u)d\overline X_u=\overline\gamma_x(\overline X_t)$; see Equation (10.44) in Kyprianou \cite{K14}.

\item[(ii)] \begin{equation}
\tau_{0}^{U,-}=\tau_{\gamma_x},\quad a.s., \label{time-ii}
\end{equation}
since $U_t=X_t-\gamma_x(\overline X_t)$ and $\{U_t<0\}=\{X_t<\gamma_x(\overline X_t)\}$.

\item[(iii)]
\begin{equation}
\sigma_{f}=\tau_{f^*},\quad a.s., \label{time-iii}
\end{equation}
since $\left\{ U_{t}<f(\overline{U}_{t}) \right\}=\left\{X_t< f(\overline\gamma_x(\overline X_t))+ \gamma_x(\overline X_t)\right\}=\left\{X_t< f^*(\overline X_t)\right\}$, where
\begin{equation}
f^*(z):=f(\overline \gamma_x(z))+\gamma_x(z). \label{f**}
\end{equation}

\end{itemize}

\begin{theorem}
\label{thm tax} Consider an upward skip-free time-homogeneous Markov process
$X$ such that Assumption \ref{asp} holds, and its general tax process $U$ is defined in (\ref{taxmodel}).
For $q,s\geq0$ and $x<K\in
\mathbb{R}$, we have
\begin{align}
\mathbb{E}_x[e^{-q\tau_{K}^{U,+}}1_{\{\tau_K^{U,+}<\sigma_f\}}] &=\exp\left \{ -\int_x^{\overline \gamma_x^{-1}(K)} b_{f^*}^{(q)}(z)dz\right \}, \label{taxexit}\\
\mathbb{E}_{x}\left[  e^{-q\sigma_{f}-sY^U_{\sigma_{f}}}1_{\{\overline
{U}_{\sigma_{f}}\leq K\}}\right]   & =\int_{x}^{\overline\gamma_x^{-1}(K)}e^{-\int_{x}^{y}b_{f^*}^{(q)}(z)\mathrm{d}z}c_{f^*}^{(q,s)}(y)\mathrm{d}y, \label{taxdd}
\end{align}
with $f^*(\cdot)$ given in (\ref{f**}).

\end{theorem}

\begin{proof} Using time correspondences (\ref{time-i}) and (\ref{time-iii}), as well as Equation (\ref{exitK}), one finds
\begin{align*}
\mathbb{E}_x[e^{-q\tau_K^{U,+}}1_{\{\tau_K^{U,+}<\sigma_{f}\}}]=&\mathbb{E}_x[e^{-q\tau_{\overline\gamma_x^{-1}(K)}^+}1_{\{\tau_{\overline\gamma_x^{-1}(K)}^+<\tau_{f^*}\}}]=\exp\left \{ -\int_x^{\overline \gamma_x^{-1}(K)} b_{f^*}^{(q)}(z)dz\right \},
\end{align*}
which proves (\ref{taxexit}).

Similarly, noting $\left\{\sigma_{f}, Y^U_{\sigma_{f}}, \overline{U}_{\sigma_{f}} \right\}\stackrel{d}{=} \left\{\tau_{f^*}, Y_{\tau_{f^*}}, \overline \gamma_x(\overline {X}_{f^*})  \right\}$, we have
\begin{align*}
\mathbb{E}_{x}\left[  e^{-q\sigma_{f}-sY^U_{\sigma_{f}}}1_{\{\overline
{U}_{\sigma_{f}}\leq K\}}\right]&=\mathbb{E}_{x}\left[e^{-q\tau_{f^*}-sY_{\tau_{f^*}}}1_{\{\overline \gamma_x(\overline {X}_{f^*})\leq K\}}\right]\\
& =\mathbb{E}_{x}\left[e^{-q\tau_{f^*}-sY_{\tau_{f^*}}}1_{\{\overline {X}_{f^*}\leq\overline \gamma_x^{-1}(K)\}}\right]\\
&=\int_{x}^{\overline\gamma_x^{-1}(K)}e^{-\int_{x}^{y}b_{f^*}^{(q)}(z)\mathrm{d}z}c_{f^*}^{(q,s)}(y)\mathrm{d}y.
\end{align*}

\end{proof}

In the following proposition, we provide the results relating to the expected present value of tax up to some certain stopping times. We denote $\eta(\cdot)$ as a general tax payment function, which depends on the surplus level at the moment of paying tax.

\begin{proposition}\label{prop1}
For $x<K$ and any function $\eta(\cdot)>0$, the expected present value of tax until general drawdown or exiting above is
\begin{align*}
&\mathbb{E}_x\left[\int_0^{\tau_K^{U,+}\wedge \sigma_{f}}e^{-q u}\eta(\overline X_u) d\overline X_u\right]=\int_x^{\overline\gamma_x^{-1}(K)}\eta(y) \exp\left \{ -\int_x^{y} b_{f^*}^{(q)}(z)dz\right \}dy,
\end{align*}
and the expected present value of tax until reaching level $K$ before general drawdown is
\begin{align*}
&\mathbb{E}_x\left[\int_0^{\tau_K^{U,+}}e^{-q u}\eta(\overline X_u) d\overline X_u1_{\{\tau_K^{U,+}<\sigma_{f}\}}\right]=\int_x^{\overline\gamma_x^{-1}(K)}\eta(y)  \exp\left \{ -\int_x^{y} b_{f^*}^{(q)}(z)dz -\int_y^{\overline\gamma_x^{-1}(K)} b_{f^*}^{(0)}(z)dz\right \}dy.
\end{align*}
\end{proposition}
\begin{proof} Thanks to the path/time correspondences in (\ref{time-i})-(\ref{time-iii}), we have
\begin{align*}
\mathbb{E}_x\left[\int_0^{\tau_K^{U,+}\wedge \sigma_{f}}e^{-q u}\eta(\overline X_u) d\overline X_u\right]= &\mathbb{E}_x\left[\int_0^{\tau_{\overline\gamma_x^{-1}(K)}^+\wedge\tau_{f^*}\wedge e_q }\eta(\overline X_u) d\overline X_u\right]\\
 =&\int_x^{\infty}\int_{x}^z \eta(y)dy \mathbb{P}_x(\overline X(\tau_{\overline\gamma_x^{-1}(K)}^+\wedge\tau_{f^*}\wedge e_q )\in dz)\\
 =&\int_x^{\overline\gamma_x^{-1}(K)}\eta(y)  \mathbb{P}_x(\overline X(\tau_{f^*}\wedge e_q )>y)dy \nonumber \\
=&\int_x^{\overline\gamma_x^{-1}(K)}\eta(y)  \mathbb{P}_x(\tau_y^+<\tau_{f^*}\wedge e_q)dy\\
=&\int_x^{\overline\gamma_x^{-1}(K)}\eta(y)  \mathbb{E}_x(e^{-q\tau_y^+}1_{\{ \tau_y^+<\tau_{f^*}\}})dy\\
=&\int_x^{\overline\gamma_x^{-1}(K)}\eta(y) \exp\left \{ -\int_x^{y} b_{f^*}^{(q)}(z)dz\right \}dy,
\end{align*}
and
\begin{align*}
\mathbb{E}_x\left[\int_0^{\tau_K^{U,+}}e^{-q u}\eta(\overline X_u) d\overline X_u1_{\{\tau_K^{U,+}<\sigma_{f}\}}\right]=&\mathbb{E}_x\left[\int_0^{\tau_{\overline\gamma_x^{-1}(K)}^+\wedge e_q }\eta(\overline X_u) d\overline X_u1_{\{\tau_{\overline\gamma_x^{-1}(K)}^+< \tau_{f^*} \}}\right]\\
 =&\int_x^{\infty}\int_{x}^z \eta(y)dy \mathbb{P}_x[\overline X(\tau_{\overline\gamma_x^{-1}(K)}^+\wedge e_q)\in dz, \tau_{\overline\gamma_x^{-1}(K)}^+< \tau_{f^*} ]\\
 =&\int_x^{\infty}\eta(y)  \mathbb{P}_x[\overline X(\tau_{\overline\gamma_x^{-1}(K)}^+\wedge e_q)>y, \tau_{\overline\gamma_x^{-1}(K)}^+< \tau_{f^*}]dy \nonumber \\
=&\int_x^{\infty}\eta(y)  \mathbb{P}_x(\tau_y^+<\tau_{\overline\gamma_x^{-1}(K)}^+\wedge\tau_{f^*} \wedge e_q)\mathbb{P}_y ( \tau_{\overline\gamma_x^{-1}(K)}^+< \tau_{f^*})dy\\
=&\int_x^{\overline\gamma_x^{-1}(K)}\eta(y)  \exp\left \{ -\int_x^{y} b_{f^*}^{(q)}(z)dz\right \}\exp\left \{ -\int_y^{\overline\gamma_x^{-1}(K)} b_{f^*}^{(0)}(z)dz\right \}dy,
\end{align*}
which completes the proof.
\end{proof}

\begin{remark} In a special case with $\gamma$ being a constant, we have
\begin{equation*}
\overline \gamma_x(y)=y-\gamma(y-x), \mbox{ and } \overline \gamma_x^{-1}(y)=\frac{y-\gamma x}{1-\gamma},
\end{equation*}
and
\begin{equation*}
f^*(z)=f(\overline \gamma_x(z))+z- \overline \gamma_x(z)=f(z-\gamma(z-x))+\gamma(z-x).
\end{equation*}
Rewriting (\ref{taxexit}) using a change of variable,
\[
\mathbb{E}_x\left [e^{-q\tau_{K}^{U,+}}1_{\{\tau_K^{U,+}<\sigma_f\}}\right]=\exp\left \{ -\int_x^{\overline \gamma_x^{-1}(K)} b_{f^*}^{(q)}(z)\mathrm{d}z\right \}=\exp\left \{ -\int_x^{K}\frac{1}{1-\gamma} b_{f^*}^{(q)}(\overline \gamma_x^{-1}(y))\mathrm{d}y\right \},
\]
 Introducing
 \begin{equation*}
  W_{f^*}^{(q)}(z)=e^{\int_{z_0}^{z}b_{f^*}^{(q)}(\overline \gamma_x^{-1}(y))\mathrm{d}y} \Leftrightarrow b_{f^*}^{(q)}(\overline \gamma_x^{-1}(z))= \frac{W_{f^*}^{(q)\prime}(z)}{W_{f^*}^{(q)}(z)},
\end{equation*}
for some fixed $x_0 \leq K$, and we may rewrite (\ref{taxexit}) as
\begin{equation*}
\mathbb{E}_{x}\left[  e^{-q\tau_K^{U,+}}1_{\{ \tau_K^{U,+} < \sigma_{f} \}}\right]=\left(\frac{W_{f^*}^{(q)}(x)}{W_{f^*}^{(q)}(K)}\right)^{\frac 1{1-\gamma}}.
\end{equation*}
Thus, the multiplicative structure 
is still present {with  generalized  drawdown  times}, and tax introduces an extra power, see, e.g., \cite{AH07} and \cite{AI14}.

\end{remark}

\section{Examples}

In this section, we consider the Spectrally Negative L\'evy process, time-homogeneous diffusion process and Ornstein-Uhlenbeck process with exponential jumps for specific examples. These processes are of particular interests thanks to their various applications in insurance and finance.

\subsection{Spectrally negative L\'evy process}

Consider a spectrally negative L\'{e}vy process $X$. Let $\psi(s):=\frac{1}%
{t}\log\mathbb{E}[e^{sX_{t}}]$, $ s\geq0$, be the Laplace
exponent of $X$. Further, let $W^{(q)}:%
\mathbb{R}
\rightarrow\lbrack0,\infty)$ be the well-known $q$-scale function of $X$. The
second scale function is defined as $Z^{(q)}(x)=1+q\int_{0}^{x}W^{(q)}%
(y)\mathrm{d}y$. We assume the scale functions are continuously
differentiable. For $p=q-\psi(s)$, let $W_{s}^{(p)}$ ($Z_{s}^{(p)}$) be the
(second) scale function of $X$ under a new probability measure\ $\mathbb{P}%
^{s}$ defined by the Radon-Nikodym derivative process $\left.  \frac
{\mathrm{d}\mathbb{P}^{s}}{\mathrm{d}\mathbb{P}}\right\vert _{\mathcal{F}_{t}%
}=e^{sX_{t}-\psi(s)t}$ for $t\geq0$. Recall that
\[
B^{(q)}(x;u,v)=\mathbb{E}_{x}\left[  e^{-q\tau_{v}^{+}}1_{\left\{  \tau_{v}%
^{+}<\infty,\tau_{v}^{+}<\tau_{u}^{-}\right\}  }\right]  =\frac{W^{(q)}%
(x-u)}{W^{(q)}(v-u)},
\]
and%
\[
C^{(q,s)}(x;u,v)=\mathbb{E}_{x}\left[  e^{-q\tau_{u}^{-}-s(u-X_{\tau_{u}^{-}}%
)}1_{\left\{  \tau_{u}^{-}<\infty,\tau_{u}^{-}<\tau_{v}^{+}\right\}  }\right]
=Z_{s}^{(p)}(x-u)-Z_{s}^{(p)}(v-u)\frac{W_{s}^{(p)}(x-u)}{W_{s}^{(p)}(v-u)}.%
\]

It is direct to check that Assumption \ref{asp} is satisfied. More
specifically,
\begin{align*}
b_{f}^{(q)}(x)  =-\left.  \frac{\partial B^{(q)}(x;f(x),v)}{\partial v}\right\vert _{v=x}=\frac{W^{(q)\prime}(x-f(x))}{W^{(q)}(x-f(x))},
\end{align*}
and
\begin{align*}
c_{f}^{(q,s)}(x)  & =\left.  \frac{\partial C^{(q,s)}(x;f(x),v)}%
{\partial v}\right\vert _{v=x} =Z_{s}^{(p)}(x-f(x))\frac{W_{s}^{(p)\prime}(x-f(x))}{W_{s}^{(p)}%
(x-f(x))}-Z_{s}^{(p)\prime}(x-f(x)).
\end{align*}
Then Theorem \ref{thm markov} implies, for $x\leq K$,
\[
\mathbb{E}_{x}\left[  e^{-q\tau_{f}-sY_{\tau_{f}}}1_{\{\tau_{f}<\infty
,\overline X_{\tau_{f}}\leq K\}}\right]  =\int_{x}^{K}e^{-\int_{x}^{y}\frac
{W^{(q)\prime}(z-f(z))}{W^{(q)}(z-f(z))}\mathrm{d}z}\left(  Z_{s}%
^{(p)}(y-f(y))\frac{W_{s}^{(p)\prime}(y-f(y))}{W_{s}^{(p)}(y-f(y))}%
-Z_{s}^{(p)\prime}(y-f(y))\right)  \mathrm{d}y,
\]
which is consistent with Proposition 3.1 in \cite{LVZ17}.

In particular, suppose that
\[
f(x)=\xi x-d\text{,}%
\]
where $\xi \leq1$ and $d>0$ are two fixed constants. One has a simplified
formula because
\[
e^{-\int_{x}^{y}\frac{W^{(q)\prime}(z-f(z))}{W^{(q)}(z-f(z))}\mathrm{d}%
z}=e^{-\int_{x}^{y}\frac{W^{(q)\prime}((1-\xi)z+d)}{W^{(q)}((1-\xi)z+d)}%
\mathrm{d}z}=e^{-\frac{1}{1-\xi}\ln W^{(q)}((1-\xi)z+d)|_{z=x}^{z=y}}=\left(
\frac{W^{(q)}((1-\xi)x+d)}{W^{(q)}((1-\xi)y+d)}\right)  ^{\frac{1}{1-\xi}}.
\]

Below is a direct corollary from Theorem \ref{thm tax} and Proposition \ref{prop1}.
\begin{corollary}\label{cor2}
For $x<K$ and any function $\eta(\cdot)$,
\begin{align*}
\mathbb{E}_x\left[e^{-q\tau_{K}^{U,+}}1_{\{\tau_K^{U,+}<\sigma_f\}}\right] &=\exp\left \{ -\int_x^{\overline \gamma_x^{-1}(K)} \frac{W^{(q)\prime}(\overline f(\overline \gamma_x(t)))}{W^{(q)}(\overline f(\overline \gamma_x(t)))} dt \right \},
\end{align*}
\begin{align*}
&\mathbb{E}_x\left[\int_0^{\tau_K^{U,+}\wedge \sigma_{f}}e^{-q u}\eta(\overline X_u) d\overline X_u\right]=\int_x^{\overline\gamma_x^{-1}(K)}\eta(y)  \exp\left(-\int_x^{y}\frac{W^{(q)\prime}(\overline f(\overline \gamma_x(t)))}{W^{(q)}(\overline f(\overline \gamma_x(t)))} dt\right)dy,
\end{align*}
and
\begin{align*}
&\mathbb{E}_x\left[\int_0^{\tau_K^{U,+}}e^{-q u}\eta(\overline X_u) d\overline X_u1_{\{\tau_K^{U,+}<\sigma_{f}\}}\right]\\
=&\int_x^{\overline\gamma_x^{-1}(K)}\eta(y)  \exp\left(-\int_x^{y}\frac{W^{(q)\prime}(\overline f(\overline \gamma_x(t)))}{W^{(q)}(\overline f(\overline \gamma_x(t)))} dt-\int_y^{\overline\gamma_x^{-1}(K)}\frac{W^{\prime}(\overline f({\overline\gamma_x}(t))}{W(\overline f({\overline\gamma_x}(t))} dt\right)dy,
\end{align*}
where $\overline f(x)=x-f(x)$.
\end{corollary}

\begin{remark}
In the special case, where $\gamma(\cdot)=\gamma$ and $f(s)=\xi s-d$, we have
\[
\overline \gamma_x(s)= s-\gamma s+\gamma x, \qquad \overline f(s)=(1-\xi) s+d,
\]
\[
\overline f(\overline \gamma_x(t))=\overline{\gamma_x+f(\overline \gamma_x)}=(1-\xi)(s-\gamma s+\gamma x)+d.
\]
Hence, Corollary \ref{cor2} reduces to
\begin{align*}
&\mathbb{E}_x\left[e^{-q\tau_K^{U,+}}1_{\{\tau_K^{U,+}<\sigma_{f}\}}\right]=\exp\left(-\frac{1}{(1-\xi)(1-\gamma)}\int_{(1-\xi)x+d}^{(1-\xi)K+d}\frac{W^{(q)\prime}(y)}{W^{(q)}(y)} dy\right)=\left( \frac{W^{(q)}((1-\xi)x+d)}{W^{(q)}((1-\xi)K+d)} \right)^{\frac{1}{(1-\xi)(1-\gamma)}},
\end{align*}
and furthermore, by letting $\eta(\cdot)=1$, we have
\begin{align*}
&\mathbb{E}_x\left[\int_0^{\tau_K^{U,+}\wedge \sigma_{f}}e^{-q u}d\overline X_u\right]=\frac{1}{1-\gamma}\int_x^{K}  \left( \frac{W^{(q)}((1-\xi)x+d)}{W^{(q)}((1-\xi)z+d)} \right)^{\frac{1}{(1-\xi)(1-\gamma)}} dz,
\end{align*}
and
\begin{align*}
&\mathbb{E}_x\left[\int_0^{\tau_K^{U,+}}e^{-q u}\gamma(\overline X_u) d\overline X_u;{\tau_K^{U,+}<\sigma_{f}}\right]=\frac{1}{1-\gamma}\int_x^{K}  \left( \frac{W^{(q)}((1-\xi)x+d)}{W^{(q)}((1-\xi)z+d)}\frac{W((1-\xi)z+d)}{W((1-\xi)a+d)} \right)^{\frac{1}{(1-\xi)(1-\gamma)}} dz.
\end{align*}
which are consistent with Theorems 1.1 and 1.2 in \cite{AVZ17} respectively.

\end{remark}

\subsection{Time-homogeneous diffusion process}

Consider a linear diffusion process $X$ of the form
\[
\mathrm{d}X_{t}=\mu(X_{t})\mathrm{d}t+\sigma(X_{t})\mathrm{d}B_{t},
\]
where $(B_{t})_{t\geq0}$ is a standard Brownian motion, and the drift term
$\mu(\cdot)$ and local volatility $\sigma(\cdot)>0$ satisfy the usual
Lipschitz continuity and linear growth conditions. The infinitesimal generator
of $X$ is given by
\[
\mathcal{L}_X=\frac{1}{2}\sigma^{2}(x)\frac{\mathrm{d}^{2}}{\mathrm{d}x^{2}}%
+\mu(x)\frac{\mathrm{d}}{\mathrm{d}x}.
\]
It is well-known that, for any $q>0$, there exist two independent and positive
solutions, denoted as $\phi_{q}^{\pm}(y)$, to the Sturm-Liouville equation
\begin{equation}
\mathcal{L}_X\phi_{q}^{\pm}(y)=q\phi_{q}^{\pm}(y), \label{SL}%
\end{equation}
where $\phi_{q}^{+}(\cdot)$ is strictly increasing and $\phi_{q}^{-}(\cdot)$
is strictly decreasing.

Thanks to $\phi_{q}^{\pm}(y)$, it is known that
\[
B^{(q)}(x;u,v)=\mathbb{E}_{x}\left[  e^{-q\tau_{v}^{+}}1_{\left\{  \tau_{v}^{+}<\tau_{u}^{-}\right\}  }\right]  =\frac{\Phi_{q}(u,x)}%
{\Phi_{q}(u,v)},
\]
and%
\[
C^{(q,s)}(x;u,v)=\mathbb{E}_{x}\left[  e^{-q\tau_{u}^{-}}1_{\left\{\tau_{u}^{-}<\tau_{v}^{+}\right\}  }\right]  =\frac{\Phi_{q}(x,v)}%
{\Phi_{q}(u,v)},
\]
where $\Phi_{q}(x,y):=\phi_{q}^{+}(x)\phi_{q}^{-}(y)-\phi_{q}^{+}(y)\phi
_{q}^{-}(x)$. Note that $C^{(q,s)}(x;u,v)$ does not depend on the argument $s$ since the diffusion process has $X_{\tau_u^-}=u$ a.s.

Then Assumption \ref{asp} is satisfied, and we have
\[
b_{f}^{(q)}(x)= -\left.  \frac{\partial B^{(q)}(x;f(x),v)}{\partial v}\right\vert _{v=x}%
=\frac{\Phi_{q,2}(f(x),x)}{\Phi_{q}(f(x),x)},
\]
and
\[
c_{f}^{(q,s)}(x)=\left.  \frac{\partial C^{(q,s)}(x;f(x),v)}%
{\partial v}\right\vert _{v=x}=\frac{\Phi_{q,2}(x,x)}{\Phi_{q}(f(x),x)}=-\frac{\Phi
_{q,1}(x,x)}{\Phi_{q}(f(x),x)},
\]
where $\Phi_{q,1}(x,y):=\frac{\partial}{\partial x} \Phi_{q}(x,y)$ and $\Phi_{q,2}(x,y):=\frac{\partial}{\partial y} \Phi_{q}(x,y)$. Notice the fact that $\Phi_{q,2}(x,x)=-\Phi_{q,1}(x,x)$.
\begin{corollary} \label{cor diff}
For $q,s\geq0$ and $x<K\in
\mathbb{R}$, we have
\begin{align*}
\mathbb{E}_{x}\left[  e^{-q\tau_{K}^{+}}1_{\{\tau_{K}^{+}<\tau_{f}%
\}}\right]    & =e^{-\int_{x}^{K}\frac{\Phi_{q,2}(f(z),z)}{\Phi_{q}(f(z),z)}\mathrm{d}z}, \\
\mathbb{E}_{x}\left[  e^{-q\tau_{f}}1_{\{\overline
{X}_{\tau_{f}}\leq K\}}\right]    & =\int_{x}^{K}e^{-\int_{x}^{y}\frac{\Phi_{q,2}(f(z),z)}{\Phi_{q}(f(z),z)}\mathrm{d}z}\frac{\Phi_{q,2}(x,x)}{\Phi_{q}(f(x),x)}\mathrm{d}y.
\end{align*}

\end{corollary}

\begin{remark}
In the special case, with proper choices of $q$ and $x=0$, it is easy to check that the results in Corollary \ref{cor diff} are consistent with Equations (20) and (21) in \cite{L77}.
\end{remark}

\subsection{Ornstein-Uhlenbeck process with exponential jumps}
Consider a generalized Ornstein-Uhlenbeck process $X$ with negative jumps, where
\[
\mathrm{d}X_{t}=\theta(\mu-X_t)\mathrm{d}t +\sigma \mathrm{d}B_t-\mathrm{d}\left(\sum_{i=1}^{N_t}P_i\right),
\]
where $\theta>0$, $\mu\in \mathbb{R}$ and $X_0=x$. Also, $(B_{t})_{t\geq0}$ is a standard Brownian motion, and $\sum_{i=1}^{N_t}P_i$ is an independent compound Poisson process. In particular, we assume the Poisson process $(N_t)_{t\geq0}$ has intensity $\lambda$, and the jumps follow the exponential distribution with mean $1/\eta$. Note that one could rewrite the process $X$ as
\begin{align*}
X_t&=X_0-\theta\int_0^{t}X_s ds+K_t,\\
X_t&=X_0e^{-\theta t} +e^{-\theta t} \int_0^t e^{\theta s} d K_s,
\end{align*}
with $K_t=(\theta \mu) t +\sigma B_t -\sum_{i=1}^{N_t}P_i$ being a Brownian perturbed Cram\'er-Lundberg process, whose Laplace exponent is
\[
\psi(s):=\frac{1}{t}\log \mathbb{E}[e^{sK_t}]=\theta \mu s +\frac{\sigma^2}{2} s^2+\lambda(\frac{\eta}{\eta+s}-1).
\]

From Lemmas 2.1 and 2.2 in \cite{ZWB17}, where the authors examined the occupation times of Ornstein-Uhlenbeck process with two-sided exponential jumps, we have the following results:
\begin{align}
\mathbb{E}_x\left[ e^{-q\tau_u^-}1_{\{X_{\tau_u^-}=u\}} \right]&=\frac{C_2^{q}(u)F_1^{q}(x)-C_1^{q}(u)F_2^{q}(x)}{C_2^{q}(u)F_1^{q}(u)-C_1^{q}(u)F_2^{q}(u)}=:I_1(x,u),\label{F1}\\
\mathbb{E}_x\left[ e^{-q\tau_u^--s(u-X_{\tau_u^-})}1_{\{X_{\tau_u^-}<u\}}\right]&=\frac{F_1^{q}(u)F_2^{q}(x)-F_2^{q}(u)F_1^{q}(x)}{C_2^{q}(u)F_1^{q}(u)-C_1^{q}(u)F_2^{q}(u)}\frac{\eta}{\eta+s}=:I_2(x,u)\frac{\eta}{\eta+s},\label{F2}\\
\mathbb{E}_x\left[ e^{-q\tau_v^+}\right]&=\frac{F_3^{q}(x)}{F_3^{q}(v)}, \label{F3}
\end{align}
where
\begin{align*}
&\phi_q(x):=|x|^{\frac{q}{\theta}-1}e^{-\frac{\sigma^2}{4\theta}x^2+\mu x}|x-\eta |^{\frac{\lambda}{\theta}},\quad F_i^{q}(x):=\int_{\Gamma_i} \phi_q(z)e^{-xz} dz, \quad C_i^{q}:=-\int_{\Gamma_i} \frac{\eta}{z-\eta} \phi_q(z)e^{-xz}dz,
\end{align*}
with $\Gamma_1=(0,\eta)$, $\Gamma_2=(\eta,\infty)$ and $\Gamma_3=(-\infty,0)$. Hence, using the strong Markov property, one has
\begin{align*}
&B^{(q)}(x;u,v)  =\mathbb{E}_{x}\left[  e^{-q\tau_{v}^{+}}1_{\left\{\tau_{v}^{+}<\tau_{u}^{-}\right\}  }\right]\\
&= \mathbb{E}_{x} \left[ e^{-q \tau_v^+} \right] -
\mathbb{E}_{x}\left[e^ {-q\tau_u^- }1_{\{X_{\tau_u^-}=u, \tau_u^-<\tau_v^+\} }\right]
\mathbb{E}_{u}\left[ e^{-q \tau_v^+}\right]
-\int_0^{\infty}\mathbb{E}_{x}\left[e^ {-q\tau_u^- }1_{\{u-X_{\tau_u^-}\in dy, \tau_u^-<\tau_v^+\} }\right] \mathbb{E}_{u-y}\left[ e^{-q \tau_v^+}\right],
 \end{align*}
 and
\begin{align*}
C^{(q,s)}(x;u,v) &=\mathbb{E}_{x}\left[  e^{-q\tau_{u}^{-}-s(u-X_{\tau_{u}^{-}})}1_{\left\{  \tau_{u}^{-}<\tau_{v}^{+}\right\}  }\right] \\
&=\mathbb{E}_{x}\left[  e^{-q\tau_{u}^{-}-s(u-X_{\tau_{u}^{-}})}\right] -\mathbb{E}_{x}\left[  e^{-q\tau_{v}^{+}}1_{\left\{\tau_{v}^{+}<\tau_{u}^{-}\right\}  }\right] \mathbb{E}_{v}\left[  e^{-q\tau_{u}^{-}-s(u-X_{\tau_{u}^{-}})}\right]  .
\end{align*}
It is easy to solve $B^{(q)}(x;u,v)$ and $C^{(q,s)}(x;u,v)$ using Equations (\ref{F1})-(\ref{F3}) (noticing that the `deficit' in (\ref{F2}) has an exponential density)
\begin{align*}
&B^{(q)}(x;u,v)=\frac{F_3^q(x)-I_1(x,u)F_3^q(u)-I_2(x,u)\int_0^{\infty} \eta e^{-\eta y}F_3^q(u-y)dy}{F_3^q(v)-I_1(v,u)F_3^q(u)-I_2(v,u)\int_0^{\infty} \eta e^{-\eta y}F_3^q(u-y)dy},\\
&C^{(q,s)}(x;u,v)=\left[I_1(x,u)+I_2(x,u)\frac{\eta}{\eta+s}\right]-B^{(q)}(x;u,v)\left[I_1(v,u)+I_2(v,u)\frac{\eta}{\eta+s}\right].
\end{align*}
Then Assumption \ref{asp} is satisfied, and we could obtain the differential exit parameters $b_{f}^{(q)}(x)= -\left.  \frac{\partial B^{(q)}(x;f(x),v)}{\partial v}\right\vert _{v=x}$ and $c_{f}^{(q,s)}(x)=\left.  \frac{\partial C^{(q,s)}(x;f(x),v)}{\partial v}\right\vert _{v=x}$. The differential calculations are omitted for conciseness and left for interested readers.


\small

\end{document}